\pdfoutput=1 
\RequirePackage[l2tabu, orthodox]{nag} 
\documentclass[a4paper,reqno,11pt]{article} 
\usepackage[stretch=10]{microtype} 
\usepackage[german,french,UKenglish]{babel}
\usepackage[T1]{fontenc} 
\usepackage[utf8]{inputenc} 
\usepackage{csquotes} 
\usepackage{letltxmacro} 
\usepackage{amsmath,amsfonts,amssymb,amsthm,mathrsfs}
\usepackage[margin=1.3in]{geometry}
\usepackage{setspace}
\usepackage{enumitem}
\usepackage{calc}
\usepackage{makebox}
\usepackage{xifthen}
\usepackage{titlesec} 
\usepackage{tikz}
\usepackage{tikz-cd} 
\usetikzlibrary{calc}
\usetikzlibrary{matrix,arrows}
\usetikzlibrary{decorations.markings}
\usepackage[lf]{Baskervaldx} 
\let\rtimess\rtimes
\usepackage[bigdelims,vvarbb]{newtxmath} 
\let\rtimes\rtimess 
\usepackage[cal=boondoxo]{mathalfa} 

\usepackage{verbatim} 
\usepackage{aliascnt} 
\usepackage[unicode,colorlinks=true,urlcolor=blue!70!black,citecolor=blue!60!black,linkcolor=blue!60!black,pdfauthor={Remy van Dobben de Bruyn},pdftitle={Constructible sheaves on toric varieties}]{hyperref} 

\titlespacing{\section}{0pt}{12pt plus 4pt minus 2pt}{0pt plus 2pt minus 2pt}
\titlespacing{\subsection}{0pt}{0pt plus 2pt minus 2pt}{0pt plus 2pt minus 2pt}
\titlespacing{\subsubsection}{0pt}{0pt plus 2pt minus 2pt}{0pt plus 2pt minus 2pt}

\titleformat{\section}[block]{\Large\bfseries\scshape\filcenter}{\thesection.}{1ex}{}
\titleformat{\subsection}{\large\scshape\filcenter}{\thesubsection}{1ex}{}

\numberwithin{equation}{section}
\newtheoremstyle{thms}{1em}{0pt}{\itshape}{}{\itshape\bfseries}{. ----}{ }{\thmname{#1} \thmnumber{#2}\normalfont\thmnote{ (#3)}}
\theoremstyle{thms}
\newaliascnt{Thm}{equation}
\newtheorem{Thm}[Thm]{Theorem}				
\aliascntresetthe{Thm}

\newaliascnt{Prop}{equation}							

\aliascntresetthe{Prop}

\newaliascnt{Lemma}{equation}						
\newtheorem{Lemma}[Lemma]{Lemma}
\aliascntresetthe{Lemma}

\newaliascnt{Cor}{equation}							
\newtheorem{Cor}[Cor]{Corollary}
\aliascntresetthe{Cor}

\newaliascnt{Conj}{equation}							

\aliascntresetthe{Conj}

\newaliascnt{Question}{equation}						

\aliascntresetthe{Question}

\newtheoremstyle{defs}{1em}{0pt}{}{}{\itshape\bfseries}{. ----}{ }{\thmname{#1} \thmnumber{#2}}
\theoremstyle{defs}
\newaliascnt{Rmk}{equation}							
\newtheorem{Rmk}[Rmk]{Remark}
\aliascntresetthe{Rmk}

\newaliascnt{Fact}{equation}							

\aliascntresetthe{Fact}

\newaliascnt{Def}{equation}							

\aliascntresetthe{Def}

\newaliascnt{Ex}{equation}							

\aliascntresetthe{Ex}

\newaliascnt{Con}{equation}							

\aliascntresetthe{Con}

\newaliascnt{Not}{equation}							

\aliascntresetthe{Not}

\newaliascnt{Setup}{equation}						

\aliascntresetthe{Setup}

\newaliascnt{Picture}{equation}						

\aliascntresetthe{Picture}

\newtheoremstyle{par}{1em}{0pt}{}{}{\itshape\bfseries}{. ----}{ }{\thmnumber{#2}}
\theoremstyle{par}
\newaliascnt{Par}{equation}							
\newtheorem{Par}[Par]{Paragraph}
\aliascntresetthe{Par}

\theoremstyle{thms}
\newtheorem{thm}{Theorem}
\newtheorem*{thm*}{Theorem}

\newtheorem*{lemma*}{Lemma}

\LetLtxMacro\oldproof\proof	
\renewcommand{\proof}[1][Proof]{\oldproof[#1]\unskip}
\LetLtxMacro\oldendproof\endproof
\renewcommand{\endproof}{\oldendproof\unskip}

\newenvironment{itemize*} 
  {\begin{itemize}
    \setlength{\itemsep}{1em}
    \setlength{\parskip}{-1em}
    \setlength{\topsep}{0pt}
    \setlength{\partopsep}{0pt}}
  {\end{itemize}}
  
\newenvironment{enumerate*}
  {\begin{enumerate}
    \setlength{\itemsep}{1em}
    \setlength{\parskip}{-1em}
    \setlength{\topsep}{0pt}
    \setlength{\partopsep}{0pt}}
  {\end{enumerate}}
  
\setlist{itemsep=0em,topsep=0cm,partopsep=0em,parsep=\lineskip}

\setlist[enumerate]{label=\normalfont(\alph*), align=left, leftmargin=3em, labelwidth=1em, itemindent=0pt, listparindent=0pt, labelindent=1em, labelsep=*}

\setlist[itemize]{leftmargin=1.3em}

\tikzcdset{arrow style=math font}

\newcommand{\Set}{\mathbf{Set}}
\newcommand{\Fin}{\mathbf{Fin}}

\newcommand{\Et}{\operatorname{\acute Et}}

\DeclareMathOperator{\Char}{char}

\newcommand{\et}{{\operatorname{\acute et}}}

\newcommand{\lc}{_{\operatorname{lc}}}
\newcommand{\cons}[1][]{_{\ifthenelse{\equal{#1}{}}{\operatorname{cons}}{#1\operatorname{-cons}}}}

\newcommand{\gp}{^{\operatorname{gp}}}
\DeclareMathOperator{\Pro}{Pro}
\DeclareMathOperator{\Hom}{Hom}
\DeclareMathOperator{\Ext}{Ext}
\DeclareMathOperator{\RHom}{RHom}

\DeclareMathOperator{\Fun}{Fun}

\DeclareMathOperator{\Map}{Map}

\DeclareMathOperator{\Sh}{Sh}

\DeclareMathOperator*{\colim}{colim}
\DeclareMathOperator{\Spec}{Spec}
\DeclareMathOperator{\Gal}{Gal}

\DeclareMathOperator{\ev}{ev}

\newcommand{\op}{^{\operatorname{op}}}
\newcommand{\cts}{^{\operatorname{cts}}}

\newcommand{\fin}{^{\operatorname{fin}}}

\newcommand{\punct}[1]{\makebox[0pt][l]{\,#1}} 

\let\OLDthebibliography\thebibliography
\renewcommand\thebibliography[1]{
  \OLDthebibliography{#1}
  \setlength{\parskip}{0pt}
  \setlength{\itemsep}{0pt plus 0.1em}
}

\begin{document}

\renewcommand{\sectionautorefname}{Section}
\renewcommand{\subsectionautorefname}{Subsection}		

\begin{center}
\vspace*{-3em}
\noindent\makebox[\linewidth]{\rule{15cm}{0.4pt}}
\vspace{-.5em}

{\LARGE{\textbf{Constructible sheaves on toric varieties}}}

\vspace{.5em}

{\textsc{Remy van Dobben de Bruyn}}

\vspace*{.25em}
\leavevmode
\noindent\makebox[\linewidth]{\rule{11cm}{0.4pt}}
\vspace{1em}
\end{center}

\renewcommand{\abstractname}{\small\bfseries\scshape Abstract}

\begin{abstract}\noindent
This paper gives an explicit computation of the category of constructible sheaves on a toric variety (with respect to the stratification by torus orbits). Over the complex numbers, this simplifies a description due to Braden and Lunts. The same computation is carried out for split toric varieties over an arbitrary field, for constructible \'etale sheaves whose restriction to each stratum is locally constant and tamely ramified. This gives the first explicit computation of an \'etale exodromy theorem in positive characteristic for a stratification over a partially ordered set of height greater than 1.
\end{abstract}

\section*{Introduction}
If $X$ is a split toric variety over a field $k$, it has a canonical stratification by torus orbits. This gives a continuous map $X \to S$, where $S$ is the set of orbits, endowed with the quotient topology. If we write $s \leq t$ if and only if $X_s \subseteq \overline{X_t}$, then the topology on $S$ agrees with the Alexandroff topology whose opens are the upwards closed sets for this partial order. In this paper, we compute the category $\Sh\cons[S]^t(X_\et)$ of $S$-constructible sheaves of sets on $X$ whose restriction to each stratum $X_s$ is (locally constant and) tamely ramified.

\begin{thm}\label{Thm main}
Let $X$ be a split normal toric variety over $k$, with the toric stratification $X \to S$ as above. There is a canonical equivalence
\[
\Sh\cons[S]^t(X_\et) \simeq \Fun\cts\big(\Pi_1^t(X,S),\Fin\big),
\]
where $\Pi_1^t(X,S)$ is the profinitely enriched category whose objects are $s \in S$ and whose morphisms are
\[
\Hom(s,t) = \begin{cases} \pi_1^t(X_{\geq s}), & s \leq t, \\ \varnothing, & s \not\leq t.\end{cases}
\]
Composition $\Hom(s,t) \times \Hom(r,s) \to \Hom(r,t)$ for $r \leq s \leq t$ is induced by multiplication in $\pi_1^t(X_{\geq r})$ under the map $\pi_1^t(X_{\geq s}) \to \pi_1^t(X_{\geq r})$. If $p = \Char k \geq 0$, then $\pi_1^t(X_{\geq s}) \cong \pi_1^t(X_s)$ is the semidirect product $(\widehat{\mathbf Z}^{(p')}(1))^{\dim X_s} \rtimes \Gal(\bar k/k)$, where $\widehat{\mathbf Z}^{(p')}$ denotes the prime-to-$p$-completion of $\mathbf Z$ (which we take to be $\widehat{\mathbf Z}$ if $p = 0$).
\end{thm}

We also prove the analogous result for the category of $S$-constructible sheaves of sets on a complex toric variety; see \autoref{Thm complex}. The answer is the same, after replacing $\pi_1^t$ by $\pi_1$ and $\widehat{\mathbf Z}^{(p')}(1)$ by $\mathbf Z$ and removing the profinite enrichment and continuity. This is a simplification of the statement and proof of a similar description obtained by Braden and Lunts \cite[Thm.\ 4.3.2]{BradenLunts}.

This result is an example of an \emph{exodromy correspondence}, computing constructible sheaves on a stratified topological space $X \to S$ as a functor category. The category $\Pi_1(X,S)$ is called the \emph{fundamental category} or \emph{exit path category} of $X \to S$. For reasonable topological spaces, a general exodromy correspondence was first proven by MacPherson, appearing in print in work of Treumann \cite[Thm.\ 1.2]{Treumann}. An $\infty$-categorical version was obtained by Lurie \cite[Thm.\ A.9.3]{LurieHA}, and there is now quite a body of literature on variants \cite{CurryPatel}, generalisations \cite{Lejay,PortaTeyssier,HainePortaTeyssier}, alternative proofs \cite{ClausenJansen,Jansen}, and computations \cite{JansenMg,JansenBS,ClausenJansen}. The computation of Braden and Lunts \cite{BradenLunts} predates the first published account of exodromy \cite{Treumann}, and is not stated in that language.

The \'etale exodromy theorem was first proven by Barwick, Glasman, and Haine \cite{BGH}, and this was used by Wolf to obtain a pro-\'etale version \cite{Wolf}. An interesting difference with the topological case is that, for \'etale sheaves, one may also take the limit over all stratifications: the resulting fundamental category is called the \emph{Galois category}, and it is now an internal category in profinite sets (the set of objects is itself a profinite set). A simplified proof of the results of \cite{BGH} and \cite{Wolf} will appear in \cite{vDdBWolf}, using condensed (higher) categories. For a fixed stratification, a very short and low-tech proof in the style of Grothendieck's Galois theory \cite[Exp.\ V, Thm.\ 4.1]{SGA1} is given in \cite{vDdB}.

In contrast with the topological case, very few computations have been carried out for the \'etale exodromy theorem. This paper is a first step towards a more systematic computation of constructible sheaves on (easy) stratified algebraic varieties, with the caveat that one often cannot even compute tame fundamental groups (the trivially stratified case). This paper is also a proof of concept that \'etale exodromy really gives explicit descriptions of categories of sheaves.

Section 1 sets up some notation for toric varieties. In section 2, we will prove the topological exodromy theorem for toric varieties, and in section 3, we modify the arguments to obtain the \'etale version over an arbitrary base field.

\subsection*{Acknowledgements} 
The author thanks Peter Haine for explaining some basic computations of exit path categories, and Katharina H\"ubner for answering a number of questions about tame fundamental groups and tame homotopy types. The author is supported by the NWO grant VI.Veni.212.204.

\section{Notation for toric varieties}
\begin{Par}
If $k$ is a field, then a \emph{split affine toric variety} over $k$ is a variety of the form $\Spec k[M]$ for some finitely generated, commutative, and cancellative monoid $M$ such that $M\gp$ is torsion-free. The main theorem restricts to the normal case, which means that $M$ is moreover saturated. A \emph{split toric variety} over $k$ is a variety built from a fan \cite{Fulton,CLS} or a separated and connected toric monoscheme \cite[Ch.\ II, \S1.9]{Ogus} in the usual manner.
\end{Par}

\begin{Par}\label{Par toric stratification}
Recall that a split toric variety $X$ over a field $k$ has a stratification by its torus orbits. If $S$ denotes the set of torus orbits, it has a partial order given by $s \leq t$ if and only if $X_s \subseteq \overline{X_t}$. Endowing $S$ with the Alexandroff topology whose open sets are the upwards closed sets, the map $X \to S$ taking a point to its corresponding orbit is continuous.

We write $S_{\leq s} = \{t \in S\ |\ t \leq s\}$, and similarly for $S_{\geq s}$, and write $X_{\leq s}$ and $X_{\geq s}$ for their respective preimages in $X$. Note that $X_{\leq s}$ is the closure $\overline{X_s}$, while $X_{\geq s}$ is the unique minimal open toric subvariety containing $X_s$, which is therefore necessarily affine. Note that $X_{\geq s}$ is called the \emph{star} of $X_s$ in \cite[\S3.5]{BradenLunts}, but we will avoid this terminology because there is a potential for confusion with the toric subvariety corresponding to the star of a given cone $s$, which is the closure $X_{\leq s} = \overline{X_s}$ \cite[Prop.\ 3.2.7]{CLS}.
\end{Par}

\begin{Par}
If $k = \mathbf C$, we write $\Sh\cons[S](X)$ for the category of $S$-constructible sheaves of sets on $X$, i.e.\ sheaves whose restriction to each stratum $X_s$ is locally constant. Likewise, write $\Sh\lc(X)$ for the locally constant sheaves on $X$. Recall that in topology, there is no hypothesis on finiteness of the stalks of a constructible or locally constant sheaf.

If $k$ is any field, we write $\Sh\cons[S](X_\et)$ (resp.\ $\Sh\cons[S]^t(X_\et)$) for the category of sheaves of sets on $X_\et$ whose restriction to each stratum $X_s$ is locally constant with finite stalks (resp.\ locally constant with finite stalks and tamely ramified at infinity). Likewise, write $\Sh\lc(X_\et)$ (resp.\ $\Sh\lc^t(X_\et)$) for the (tamely ramified) finite locally constant sheaves.
\end{Par}

\begin{Par}\label{Par affine}
If $X = \Spec k[M]$ is a split affine toric variety, there is a canonical exact sequence of monoids $0 \to M^\times \to M \to \overline M \to 0$, where $M^\times$ is the subgroup of invertible elements in $M$ (the \emph{group of units} of $M$) \cite[Ch.\ I, \S1.3]{Ogus}. If $Y = \Spec k[M^\times]$, there is a natural projection $\pi \colon X \twoheadrightarrow Y$ induced by $M^\times \hookrightarrow M$. It has a section $i \colon Y \hookrightarrow X$ induced by the map $M \to M^\times \cup \{0\}$ taking $M \setminus M^\times$ to $0$, inserting $Y$ as the unique closed stratum of $X$.

By \ref{Par toric stratification}, this applies in particular to the affine toric variety $X_{\geq s}$ inside any toric variety $X$, where we take $Y$ to be $X_s \hookrightarrow X_{\geq s}$. This gives a retraction $\pi \colon X_{\geq s} \to X_s$ to the inclusion.
\end{Par}

\begin{Lemma}\label{Lem A1 homotopy}
Let $X = \Spec k[M]$ be a split affine toric variety, and let $Y$ be as in \ref{Par affine}. Then the maps $i \colon Y \hookrightarrow X$ and $\pi \colon Y \twoheadrightarrow X$ are mutually inverse naive $\mathbf A^1$-homotopy equivalences.
\end{Lemma}

\begin{proof}
By \cite[Ch.\ I, Prop.\ 2.2.1]{Ogus}, we may choose a homomorphism $\phi \colon M \to \mathbf N$ with $\ker \phi = M^\times$. Then the map $M \to M \oplus \mathbf N$ given by $m \mapsto (m,\phi(m))$ gives a map $h \colon \mathbf A^1 \times X \to X$ whose restriction to $0 \times X$ is $i \circ \pi$ and whose restriction to $1 \times X$ is the identity on $X$. The other composition $\pi \circ i$ is the identity on $Y$.
\end{proof}

Thus, $\pi$ induces an isomorphism on fundamental groups, at least in the topological setting. The same should hold for tame fundamental groups in the \'etale setting, but $\mathbf A^1$-invariance of tame fundamental groups of singular varieties does not yet appear in the literature. Luckily, in this case, the statement will roll out of the proof of \autoref{Lem toric}.

\begin{Par}\label{Par section}
If $\pi \colon X \to Y$ is a morphism of schemes with a section $i \colon Y \to X$, then there is a canonical natural transformation $\pi_* \to i^*$ of functors $\Sh(X_\et) \to \Sh(Y_\et)$. Indeed, recall that $(i^*\mathscr F)(V)$ is the sheafification of the colimit of $\mathscr F(W)$ over pairs $(V \to W \to X)$ of an \'etale neighbourhood $W \to X$ and a map $V \to W$. Taking $W = \pi^{-1}(V) = V \times_Y X$ gives such a neighbourhood, defining a natural transformation $\pi_* \to i^*$. The same happens in the complex topology.
\end{Par}

\section{Complex toric varieties}
We first prove the analogue of \autoref{Thm main} for complex toric varieties (via topological fundamental groups). This is very close to the result of Braden and Lunts \cite[Thm.\ 4.3.2]{BradenLunts}. In fact, in the topological case, the method will not require knowing a priori that an exodromy theorem holds: it rolls out of the proof together with the following easy lemma (which is a variant of one of the key tools in \cite{vDdBWolf}).

\begin{Lemma}\label{Lem basis}
Let $\mathscr X$ be a topos, and assume there exists a small full subcategory $\mathscr C \subseteq \mathscr X$ that generates $\mathscr X$ such that, for every effective epimorphism $\coprod_{i \in I} V_i \twoheadrightarrow U$ with $U\in \mathscr C$, there exists $i \in I$ and a section $U \to V_i$. Then the Yoneda functor $\mathscr X \to \Fun(\mathscr C\op,\Set)$ is an equivalence.
\end{Lemma}

This is similar to the classical lemma that sheaves on a topological space can be computed as sheaves on a basis: the objects in $\mathscr C$ correspond to the opens in the basis, and the hypothesis on existence of sections means that there is no Grothendieck topology to consider. (The analogous statement also holds for hypercomplete $\infty$-topoi, as an arbitrary object has a hypercover by objects in a given generating set.)

\begin{proof}[Proof of Lemma]
As in \cite[Exp.\ IV, 1.2.3]{SGA4I}, form the smallest full subcategory $\mathscr D$ of $\mathscr X$ that contains $\mathscr C$ and is closed under finite limits. It is still a small category since $\mathscr X$ is locally small, and the Yoneda functor $\mathscr X \to \Sh_{\operatorname{can}}(\mathscr D)$ is an equivalence [\emph{loc.\ cit.}]. We have an inclusion functor $i \colon \mathscr C \to \mathscr D$, where we equip $\mathscr D$ with the canonical topology and $\mathscr C$ with the trivial topology. It suffices to show that $i$ satisfies the hypotheses of \cite[Tag \href{https://stacks.math.columbia.edu/tag/03A0}{03A0}]{Stacks}. Continuity \cite[Tag \href{https://stacks.math.columbia.edu/tag/00WV}{00WV}]{Stacks} holds vacuously, since the only covers in $\mathscr C$ are isomorphisms $\{V \stackrel\sim\to U\}$. Cocontinuity \cite[Tag \href{https://stacks.math.columbia.edu/tag/00XJ}{00XJ}]{Stacks} holds by the hypothesis on existence of sections. Conditions (3) and (4) are automatic since $i$ is fully faithful. Finally, if $\mathscr F\in \mathscr X$, then the natural map $\coprod_{U \in \mathscr C} \coprod_{\alpha \in \Hom(U,\mathscr F)} U \to \mathscr F$ is an effective epimorphism since $\mathscr C$ generates $\mathscr X$, which proves (5).
\end{proof}

This gives the following way to prove exodromy theorems (which is similar to the key idea of \cite{HainePortaTeyssier} since objects representing a fibre functor are exactly the atomic objects):

\begin{Cor}\label{Cor exodromy}
Let $X$ be a topological space, let $S$ be a finite poset endowed with the Alexandroff topology, and let $X \to S$ be a continuous map with connected fibres. For each $s \in S$, choose a point $x_s \in X_s$, and assume that each of the fibre functors $F_{x_s} \colon \Sh\cons[S](X) \to \Set$ is representable. Let $\Pi_1(X,S)$ be the full subcategory of $\Fun\big(\Sh\cons[S](X),\Set\big)$ on the functors $F_{x_s}$. Then the evaluation functor
\begin{align*}
\ev \colon \Sh\cons[S](X) &\to \Fun\big(\Pi_1(X,S),\Set\big) \\
\mathscr G &\mapsto \big(F_{x_s} \mapsto F_{x_s}(\mathscr G)\big).
\end{align*}
is an equivalence.
\end{Cor}

\begin{proof}
For each $s \in S$, let $\mathscr F_s$ be a constructible sheaf representing the fibre functor $F_{x_s}$. Let $\mathscr X$ be the topos $\Sh\cons[S](X)$ of $S$-constructible sheaves, and let $\mathscr C$ be the full subcategory on the sheaves $\mathscr F_s$. By the Yoneda lemma, the functor $\mathscr C\op \to \Fun(\mathscr X,\Set)$ identifies $\mathscr C\op$ with $\Pi_1(X,S)$ (in particular, $\Pi_1(X,S)$ is small). Since each stratum $X_s$ is connected, any fibre functor $F_y$ for $y \in X_s$ is isomorphic to $F_{x_s}$, so the functors $F_{x_s}$ are jointly conservative. This means that the sheaves $\mathscr F_s$ generate $\mathscr X$ \cite[Exp.\ I, Def.\ 7.1 and Exp.\ II, 4.9]{SGA4I}. If $\coprod_{i \in I} \mathscr G_i \twoheadrightarrow \mathscr F_s$ is an epimorphism, then the map
\[
\coprod_{i\in I} \Hom(\mathscr F_s,\mathscr G_i) = \coprod_{i \in I} F_{x_s}\big(\mathscr G_i\big) = F_{x_s}\Big( \coprod_{i \in I} \mathscr G_i \Big) \to F_{x_s}(\mathscr F_s) = \Hom(\mathscr F_s,\mathscr F_s)
\]
is surjective. Applying this to the identity $\mathscr F_s \to \mathscr F_s$ gives an element $i \in I$ and a section $\mathscr F_s \to \mathscr G_i$, so the result follows from \autoref{Lem basis}.
\end{proof}

\begin{Rmk}
This is very similar to how fundamental groups are often computed: if $X$ is path connected, locally path connected, and semi-locally simply connected, then the universal cover $\widetilde X \to X$ represents the fibre functor $F_x \colon \Sh\lc(X) \to \Set$. The fundamental group is computed as the fibre $F_x(\widetilde X) = \Hom(h_{\widetilde X},F_x) = \Hom(F_x,F_x)$, so the full subcategory on $F_x$ is the one-object category with endomorphisms given by $\pi_1(X)$.
\end{Rmk}

The computations of exodromy for toric varieties will work in the same way, representing the fibre functors $\Sh\cons[S](X) \to \Set$ by explicit constructible sheaves (or rather, their espaces \'etal\'es, giving a geometric interpretation analogous to the universal cover).

The key computation is the following lemma, whose analogue for the derived category of constructible sheaves is \cite[Lem.\ 4.1.1]{BradenLunts}. Note that a vast generalisation \cite[Lem.\ A.9.12]{LurieHA} is also one of the key local computations in Lurie's $\infty$-categorical exodromy theorem for conically stratified spaces. What allows us to compute $\Pi_1(X,S)$ globally in the toric case is that the stratification is `globally conical'.

\begin{Lemma}\label{Lem BL}
Let $X$ be a normal toric variety over $\mathbf C$.
\begin{enumerate}
\item\label{item C constructible} A sheaf $\mathscr F \in \Sh(X)$ is $S$-constructible if and only if $\mathscr F$ is a finite limit of sheaves of the form $j_{s,*}\mathscr G$, where $j_s \colon X_s \hookrightarrow X$ is some (locally closed) stratum and $\mathscr G \in \Sh\lc(X_s)$.
\item\label{item C affine} If $X$ is affine and $i \colon Y \hookrightarrow X$ and $\pi \colon X \twoheadrightarrow Y$ are as in \ref{Par affine}, then the natural transformation $\pi_* \to i^*$ of \ref{Par section} is an isomorphism when restricted to $S$-constructible sheaves.
\end{enumerate}
\end{Lemma}

We include a proof because the \'etale case (\autoref{Lem toric} and \autoref{Cor toric}) will follow the same strategy (where the argument is more difficult as we do not have access to small open balls).

\begin{proof}
Suppose $X$ is affine, and let $j_s \colon X_s \hookrightarrow X$ be the inclusion of a stratum. We first prove \ref{item C affine} for $\mathscr F = j_{s,*} \mathscr G$ for $\mathscr G \in \Sh\lc(X_s)$ (we do not yet know that $\mathscr F$ is $S$-constructible). The result only depends on the pushforward of $\mathscr G$ to $X_{\leq s}$, so we may assume $X_s = U$ is the open stratum, with inclusion $j \colon U \hookrightarrow X$. If $V \subseteq Y$ is open, then the set of open neighbourhoods of $V$ in $X$ contains a basis consisting of opens $W \subseteq \pi^{-1}(V)$ for which the inclusion $U \cap W \hookrightarrow U \cap \pi^{-1}(V)$ is a homotopy equivalence. Thus, in the colimit
\[
(i^*j_*\mathscr G)(V) = \colim_{\substack{\longrightarrow \\ V \subseteq W}} \mathscr G(U \cap W),
\]
there is a cofinal system where the restriction map $\mathscr G(U \cap \pi^{-1}(V)) \to \mathscr G(U \cap W)$ is a bijection, so we conclude that the map $(\pi_*j_*\mathscr G)(V) \to (i^*j_*\mathscr G)(V)$ is a bijection. This proves \ref{item C affine} for $j_{s,*} \mathscr G$.

Now let $X$ be arbitrary, and let $\mathscr F = j_{s,*}\mathscr G$ as in \ref{item C constructible}. If $t \in S$, we need to show that $j_t^*j_{s,*} \mathscr G$ is locally constant on $X_t$. If $t \leq s$, we may replace $X$ by $X_{\geq t} \cap X_{\leq s}$, and then the case treated above shows that $j_t^*j_{s,*}\mathscr G = (\pi \circ j_s)_*\mathscr G$, which is locally constant since $\pi \circ j_s \colon X_s \to X_t$ is a fibre bundle and $\mathscr G$ is locally constant. If $t \not\leq s$, we have $j_t^*j_{s,*}\mathscr G = *$, which is locally constant. As $\Sh\cons[S](X)$ is stable under finite limits, this shows one inclusion. The converse follows by induction on $\lvert S \rvert$: if $i \colon Z \hookrightarrow X$ is any minimal stratum and $j \colon U \hookrightarrow X$ is the complement, then every $S$-constructible sheaf $\mathscr F$ sits in a pullback square
\[
\begin{tikzcd}
\mathscr F \ar{r}\ar{d} & j_*j^*\mathscr F \ar{d} \\
i_*i^*\mathscr F \ar{r} & i_*i^*j_*j^* \mathscr F\punct{,}
\end{tikzcd}
\]
where we know that $j^*\mathscr F$ is a finite limit of sheaves of the form $j_{s,*}\mathscr G$ by induction. Then $i^*j_*$ takes $j^*\mathscr F$ to a locally constant sheaf since it preserves finite limits and takes each $j_{s,*}\mathscr G$ to a locally constant sheaf, which finishes the proof of \ref{item C constructible}. Then \ref{item C affine} follows since it holds for  $j_{s,*}\mathscr G$ and $\pi_*$ and $i^*$ preserve finite limits.
\end{proof}

\begin{Thm}\label{Thm complex}
Let $X$ be a normal toric variety over $\mathbf C$, and let $X \to S$ be the toric stratification. There is a canonical equivalence
\[
\Sh\cons[S](X) \simeq \Fun\big(\Pi_1(X,S),\Set\big),
\]
where $\Pi_1(X,S)$ is the category whose objects are $s \in S$ and whose morphisms are
\[
\Hom(s,t) = \begin{cases} \pi_1(X_{\geq s}), & s \leq t, \\ \varnothing, & s \not\leq t.\end{cases}
\]
Composition $\Hom(s,t) \times \Hom(r,s) \to \Hom(r,t)$ for $r \leq s \leq t$ is induced by multiplication in $\pi_1(X_{\geq r})$ under the map $\pi_1(X_{\geq s}) \to \pi_1(X_{\geq r})$. Moreover, $\pi_1(X_{\geq s}) \cong \pi_1(X_s) \cong \mathbf Z^{\dim X_s}$.
\end{Thm}

\begin{proof}\renewcommand{\qedsymbol}{}
For each $s \in S$, the fibre functor $\Sh\lc(X_s) \to \Set$ is represented by the universal cover $\widetilde{X_s} \to X_s$. Write $i \colon X_s \to X_{\geq s}$ for the closed immersion, and $j = j_{\geq s} \colon X_{\geq s} \to X$ for the open immersion, so that $j_s = j\circ i$. If $\pi \colon X_{\geq s} \to X_s$ is the map of \ref{Par affine}, then the pullback $i^* \colon \Sh\cons[S](X_{\geq s}) \to \Sh\lc(X_s)$ has a left adjoint $\pi^*$ by \autoref{Lem BL}. Moreover, $j_!$ is a left adjoint to $j^*$. Thus, for every $\mathscr F \in \Sh\cons[S](X)$, we have
\[
F_{x_s}(\mathscr F) = F_{x_s}(i^*j^*\mathscr F) = \Hom_{X_s}\big(\widetilde{X_s},i^*j^*\mathscr F\big) = \Hom_X\big(j_!\pi^*\widetilde{X_s},\mathscr F\big),
\]
so $F_{x_s}$ is represented by $j_!\pi^*\widetilde{X_s}$. The exodromy correspondence follows by \autoref{Cor exodromy}, and it remains to compute the $\Hom$ sets in $\Pi_1(X,S)$. Since $\pi \colon X_{\geq s} \to X_s$ is a deformation retraction of the inclusion $i \colon X_s \hookrightarrow X_{\geq s}$ by \autoref{Lem A1 homotopy}, it induces an isomorphism fundamental groups, so $\pi^*\widetilde{X_s}$ is the universal cover $\widetilde{X_{\geq s}}$. Now we conclude by the Yoneda lemma: if $s, t \in P$, then
\begin{align*}
\Hom\big(F_{x_s},F_{x_t}\big) = \Hom\big(h_{j_{\geq s,!}\widetilde{X_{\geq s}}},F_{x_t}\big) = F_{x_t}\big(j_{\geq s,!}\widetilde{X_{\geq s}}\big) = \begin{cases}\pi_1(X_{\geq s}), & s \leq t, \\ \varnothing, & s \not\leq t. \end{cases} \tag*{$\square$}
\end{align*}
\end{proof}

\begin{Rmk}\label{Rmk infty}
The same proof works $\infty$-categorically as well (with basically no modifications), showing that $\Pi_\infty(X,S) = \Pi_1(X,S)$. Indeed, \autoref{Lem BL} carries through unchanged, and the universal cover $\widetilde{X_s} \to X_s$ still represents the fibre functor $\Sh\lc(X_s,\mathscr S) \to \mathscr S$ since $\widetilde{X_s}$ is contractible. The method of \autoref{Lem basis} and \autoref{Cor exodromy} only gives an exodromy correspondence for $S$-constructible sheaves in the $\infty$-topos of hypersheaves on $X$, but Lurie's version \cite[Thm.\ A.9.3]{LurieHA} upgrades this to the full $\infty$-topos. The question of hypercompletion was studied in more detail in \cite[\S2.4]{HainePortaTeyssier}.
\end{Rmk}

\section{\'Etale sheaves on split toric varieties}
We now modify the arguments of the previous section to \'etale sheaves on split toric varieties over an arbitrary field $k$. In short, the main differences in the proof are:
\begin{itemize}
\item The computation of \autoref{Lem BL}\ref{item C affine} in the case $\mathscr F = j_{s,*} \mathscr G$ requires a different argument, which is the hardest part of the proof.
\item The fibre functors $\Sh\cons[S]^t(X_\et) \to \Fin$ are no longer representable, but merely pro-representable. Thus, instead of using \autoref{Cor exodromy}, we need to know a priori that an exodromy theorem holds, which is provided by \cite{vDdB} (it can also be deduced from \cite{BGH}; see the introduction of \cite{vDdB} for a comparison).
\end{itemize}

\begin{Par}
If $\Char k = p \geq 0$, write $\widehat{\mathbf Z}^{(p')}$ for the limit of $\mathbf Z/n\mathbf Z$ over all $n \in \mathbf Z_{>0}$ invertible in $k$. If $p = 0$, it is the profinite completion $\widehat{\mathbf Z}$, and if $p > 0$, it is the usual prime-to-$p$ completion. Write $\mathbf Z_{p'}$ for the subring of $\mathbf Q$ obtained by inverting all primes invertible in $k$. Thus, $\mathbf Z_{p'} = \mathbf Q$ if $p = 0$ and $\mathbf Z_{p'} = \mathbf Z_{(p)}$ if $p > 0$.
\end{Par}

\begin{Par}\label{Par tame G_m}
If $X = \mathbf G_m^n$ over some algebraically closed field $k$ of characteristic $p \geq 0$, then $\pi_1^t(X) \cong \big(\widehat{\mathbf Z}^{(p')}\big)^n$. For $n = 1$, this follows easily from the Riemann--Hurwitz theorem, or by lifting to characteristic $0$ \cite[Exp.\ XIII, Cor.\ 2.12]{SGA1}. For arbitrary $n$, we may for instance use the K\"unneth formula of \cite[Prop.\ 3]{Hoshi}, noting that the groups there agree with the tame fundamental groups by \cite[Prop.\ B.7]{Hoshi} and \cite[Thm.\ 4.4(vii)]{KerzSchmidt}. 

If $X = \Spec k[M]$ for a finitely generated free abelian group $M$, then $\pi_1^t(X)$ is canonically identified with $\widehat{M^\vee}{}^{(p')}$, where $M^\vee = \Hom(M,\mathbf Z)$ is the linear dual. For a surjection $M^\vee \twoheadrightarrow A$ to a finite group of order prime to $p$, the corresponding covering $Y \to X$ is given as follows: if $A^*$ denotes the Pontryagin dual of $A$, then the perfect dual $A^\vee = \RHom(A,\mathbf Z)$ is $A^*[-1]$, so perfect duality gives an isomorphism $\Hom(M^\vee,A) \cong \Ext(A^*,M)$. Thus, $M^\vee \to A$ corresponds to an extension $0 \to M \to N \to A^* \to 0$, and $Y \to X$ is the cover $\Spec k[N] \to \Spec k[M]$. In particular, the universal cover of $X$ is $\Spec k[M \otimes \mathbf Z_{p'}]$.
\end{Par}

We start with two lemmas on pullback functors induced by morphisms of schemes with geometrically connected generic fibre. They are stated topos-theoretically to unclutter the notation.

\begin{Lemma}\label{Lem preserves coproducts}
Let $f_* \colon \mathscr Y \to \mathscr X$ be a morphism of topoi, and let $\mathscr X_0 \subseteq \mathscr X$ be a small set of generators. Then the following are equivalent:
\begin{enumerate}
\item\label{item pushforward} Pushforward $f_*$ preserves finite coproducts;
\item\label{item decomposition} For every $U \in \mathscr X$, every decomposition $f^*U = V \amalg W$ is obtained by pullback from a unique decomposition of $U$.
\item\label{item decomposition basis} For every $U \in \mathscr X_0$, every decomposition $f^*U = V \amalg W$ is obtained by pullback from a unique decomposition of $U$.
\end{enumerate}
If $f_*$ is the map $\Sh(Y_\et) \to \Sh(X_\et)$ induced by a dominant morphism $f \colon Y \to X$ of normal integral schemes, then these are also equivalent to:
\begin{enumerate}[resume]
\item\label{item pi_0} Pullback $f^* \colon \Et_{/X}^{\operatorname{fps}} \to \Et_{/Y}^{\operatorname{fps}}$ induces bijections $\pi_0(V) \to \pi_0(f^*V)$ for all $V \in \Et_{/X}^{\operatorname{fps}}$, where $\Et{}^{\operatorname{fps}}$ denotes the wide subcategory of schemes whose morphisms are finitely presented, separated, and \'etale.
\item\label{item geometrically connected} The generic fibre of $f$ is geometrically connected.
\item\label{item primary} The function field extension $K(X) \to K(Y)$ is primary.
\end{enumerate}
\end{Lemma}

\begin{proof}
If $\mathscr F, \mathscr G \in \mathscr Y$, there is a natural map $f_*\mathscr F \amalg f_*\mathscr G \to f_*(\mathscr F \amalg \mathscr G)$, which is an isomorphism if and only if the map
\[
\Hom\big(U,f_*\mathscr F \amalg f_*\mathscr G\big) \to \Hom\big(U,f_*(\mathscr F \amalg \mathscr G)\big) \cong \Hom\big(f^*U,\mathscr F\amalg \mathscr G\big)
\]
is a bijection for all $U \in \mathscr X$, or equivalently for all $U \in \mathscr X_0$. Naturality in $\mathscr F$ and $\mathscr G$ gives a commutative diagram
\[
\begin{tikzcd}
\Hom\big(U,f_*\mathscr F \amalg f_*\mathscr G\big) \ar{r}\ar{d} & \Hom\big(f^*U,\mathscr F\amalg \mathscr G\big) \ar{d} \\
\Hom\big(U,* \amalg *\big) \ar{r} & \Hom\big(f^*U,*\amalg *\big)
\end{tikzcd}
\]
Since coproducts are disjoint and universal, elements on the bottom correspond to decompositions of $U$ and $f^*U$ respectively, and the bottom map sends $U = V \amalg W$ to $f^*U = f^*V \amalg f^*W$ by exactness of $f^*$. The fibre of the left vertical map over a decomposition $U = V \amalg W$ is $\Hom(V,f_*\mathscr F) \times \Hom(W,f_*\mathscr G)$, and likewise for the right vertical map. Thus, if the bottom map is a bijection, then so is the top map because of the isomorphism $\Hom(V,f_*\mathscr F) \times \Hom(W,f_*\mathscr G) \cong \Hom(f^*V,\mathscr F) \times \Hom(f^*W,\mathscr G)$. This proves the equivalences \ref{item pushforward} $\Leftrightarrow$ \ref{item decomposition} $\Leftrightarrow$ \ref{item decomposition basis}.

In the case of normal schemes, \ref{item pi_0} is a reformulation of \ref{item decomposition basis} since $\Et_{/X}^{\operatorname{fps}}$ forms a basis of $\Sh(X_\et)$. Since $X$ and $Y$ are normal, so is any $V$ in $\Et_{/X}$ or $\Et_{/Y}$ \cite[Tag \href{https://stacks.math.columbia.edu/tag/033C}{033C}]{Stacks}. Since $V \to X$ is flat, every generic point of an open $U \subseteq V$ lies above the generic point of $X$ \cite[Tag \href{https://stacks.math.columbia.edu/tag/00HS}{00HS}]{Stacks}, hence there are finitely many since $V \to X$ is finitely presented \'etale. We conclude that $X$ is a finite disjoint union of normal schemes \cite[Tag \href{https://stacks.math.columbia.edu/tag/0357}{0357}]{Stacks}. 

In particular, $\pi_0(V)$ agrees with $\pi_0\big(V \times_X \Spec K(X)\big)$ for $V \in \Et_{/X}^{\operatorname{fps}}$, and likewise for $Y$. Since every finite extension of $K(X)$ can be spread out to a finitely presented, separated, \'etale map $V \to X$, condition \ref{item pi_0} becomes equivalent to the statement that pullback $\Et_{/K(X)} \to \Et_{/K(Y)}$ preserves $\pi_0$, which is exactly \ref{item geometrically connected}.
The final equivalence \ref{item geometrically connected} $\Leftrightarrow$ \ref{item primary} is \cite[Tag \href{https://stacks.math.columbia.edu/tag/0G33}{0G33}]{Stacks} (which uses `geometrically irreducible', but the proof also works for `geometrically connected').
\end{proof}

\begin{Lemma}\label{Lem locally constant}
Let $f \colon \mathscr Y \to \mathscr X$ be a morphism of topoi satisfying the conditions of \autoref{Lem preserves coproducts}.
\begin{enumerate}
\item\label{item slice topos} If $U \in \mathscr X$, then the induced geometric morphism $\mathscr Y_{/f^*U} \to \mathscr X_{/U}$ also satisfies the conditions of \autoref{Lem preserves coproducts}.
\item\label{item unit} If $\mathscr F \in \mathscr X$ is finite locally constant, then the unit $\mathscr F \to f_*f^*\mathscr F$ is an isomorphism.
\end{enumerate}
\end{Lemma}

\begin{proof}
\ref{item slice topos} We will use condition \ref{item pushforward} of \autoref{Lem preserves coproducts}. The pushforward $f_{U,*} \colon \mathscr Y_{/f^*U} \to \mathscr X_{/U}$ takes $\mathscr F \to f^*U$ to the top map of the pullback square
\[
\begin{tikzcd}
V \ar{r}\ar{d} & U \ar{d} \\
f_*\mathscr F \ar{r} & f_*f^*U\punct{.}
\end{tikzcd}
\]
The coproduct of $\mathscr F \to f^*U$ and $\mathscr G \to f^*U$ in $\mathscr Y_{/f^*U}$ is given by $\mathscr F \amalg \mathscr G \to f^*U$, so the hypothesis and universality of coproducts show that $f_{U,*}$ also preserves finite coproducts.

\ref{item unit} For any $U \in \mathscr X$, the unit $\mathscr F \times U \to f_{U,*}f_U^*(\mathscr F \times U)$ of $f_U^* \dashv f_{U,*}$ applied to $\mathscr F \times U \in \mathscr X_{/U}$ is identified with the map $\mathscr F \times U \to (f_*f^*\mathscr F) \times U$. Thus, if $\{U_i \to *\}_{i \in I}$ is a jointly epimorphic family, then to check that $\mathscr F \to f_*f^*\mathscr F$ is an isomorphism, it suffices to prove this locally in each $\mathscr X_{/U_i}$. By \ref{item slice topos}, we therefore reduce to the case where $\mathscr F$ is finite constant. Since both $f^*$ and $f_*$ preserve finite coproducts by hypothesis, we reduce to $\mathscr F = *$, where the statement follows since both $f^*$ and $f_*$ preserve finite limits.
\end{proof}

This allows us to carry out the main computation, analogous to \autoref{Lem BL}\ref{item C affine}.

\begin{Lemma}\label{Lem toric}
Let $X$ be a split normal affine toric variety over a field $k$, and let $i \colon Y \hookrightarrow X$ and $\pi \colon X \twoheadrightarrow Y$ be the maps defined in \ref{Par affine}. Let $j_s \colon X_s \hookrightarrow X$ be the inclusion of a stratum, let $\mathscr G \in \Sh\lc^t(X_{s,\et})$, and let $\mathscr F = j_{s,*} \mathscr G$. Then the map $\pi_*\mathscr F \to i^*\mathscr F$ of \autoref{Par section} is an isomorphism, and this sheaf is in $\Sh\lc^t(Y)$. Moreover, the maps $\pi_1^t(Y) \to \pi_1^t(X) \to \pi_1^t(Y)$ induced by $i$ and $\pi$ are mutually inverse isomorphisms.
\end{Lemma}

As noted after \autoref{Lem A1 homotopy}, the final statement would follow if we knew that tame fundamental groups are $\mathbf A^1$-homotopy invariant for singular varieties. However, we will see that the statement also rolls out of the proof.

\begin{proof}
The result only depends on the pushforward of $\mathscr G$ to $X_{\leq s}$, so we may assume $X_s = U$ is the open stratum and $j = j_s \colon U \hookrightarrow X$ the corresponding open immersion. To check that $\pi_*\mathscr F \to i^*\mathscr F$ is an isomorphism, it suffices to do this over the algebraic closure, so we may assume $k$ is algebraically closed. By \autoref{Lem preserves coproducts}, the functors $\pi_*$ and $j_*$ preserve finite coproducts, and the same goes for $i^*$ since it is a left adjoint. We reduce to the case that $\mathscr G$ is connected, so by \ref{Par tame G_m} it is represented by $T = \Spec k[N] \to U$ for some extension $0 \to M\gp \to N \to A^* \to 0$, corresponding to a homomorphism $\phi \colon (M\gp)^\vee \twoheadrightarrow A$.

Recall from \ref{Par affine} the short exact sequence of monoids $0 \to M^\times \to M \to \overline M \to 0$. First suppose $\phi$ does not kill $(\overline M{}\gp)^\vee$. We will show that $i^*j_* \mathscr G = \varnothing$ in this case, which forces $\pi_*j_*\mathscr G \to i^*j_*\mathscr G$ to be an isomorphism. For this, it suffices to show that if $V \to Y$ is a nonempty \'etale open and $V \to W \to X$ is an \'etale neighbourhood of $V$ in $X$, then $\Hom_U(W \times_X U,T) = \varnothing$. Suppose such a map $f \colon W \times_X U \to T$ exists. Let $\Spec K \to Y$ be the geometric generic point, choose a point $\Spec K \to V$, and let $w \colon \Spec K \to W_K$ be the resulting point. The fibre $X_K$ is the affine toric variety $\Spec K[\overline M]$ with torus $U_K = \Spec K[\overline M{}\gp]$ and closed stratum $0 = \Spec K$. Then $\Spec K \to W_K \to X_K$ is an \'etale neighbourhood of the closed stratum, so we get an isomorphism
\[
K[\![\overline M]\!] = \widehat{\mathcal O}_{X_K,0} \stackrel\sim\to \widehat{\mathcal O}_{W_K,w}.
\]
Therefore, $f$ induces a map $\widetilde f \colon \Spec K(\!(\overline M)\!) \to T_K$. Note that $T_K$ need not be connected: its components correspond to the orbits of $\overline M{}\gp$ acting on $A$ via translation. Since there exists an element of $\overline M{}\gp$ acting nontrivially on $A$, we conclude that $T_K$ splits up as a disjoint union of isomorphic copies of a nontrivial cover $T_0 \to U_K$. Since $\Gal(K(\!(\overline M)\!)) \to \pi_1^t(U_K)$ is surjective, we conclude that $T_K \times_U \Spec K(\!(\overline M)\!)$ also splits as a disjoint union of isomorphic covers of a nontrivial field extension, contradicting the fact that we have a section $\widetilde f$.

So we are left with the case where $(\overline M{}\gp)^\vee$ acts trivially on $A$, so the action descends to an action $(M^\times)^\vee \to A$. This means exactly that $\mathscr G = (\pi|_U)^* \mathscr H = j^*\pi^* \mathscr H$ for some $\mathscr H \in \Sh\lc^t(Y)$. By \autoref{Lem locally constant}\ref{item unit}, the units $\mathscr H \to \pi_*\pi^* \mathscr H$ and $\pi^*\mathscr H \to j_*j^*\pi^* \mathscr H = j_*\mathscr G$ are isomorphisms. We conclude that $\mathscr F = j_*\mathscr G = \pi^*\mathscr H$, and the map $\pi_*\mathscr F \to i^*\mathscr F$ is inverse to the unit $\mathscr H \stackrel\sim\to \pi_*\pi^*\mathscr H$. This also shows that $\pi_*\mathscr F = \mathscr H$ is in $\Sh\lc^t(Y)$.

For the statement on tame fundamental groups, we have to show that the pullback functor $\pi^* \colon \Sh\lc^t(Y_\et) \to \Sh\lc^t(X_\et)$ is an equivalence of categories. Fully faithfulness follows from \autoref{Lem locally constant}\ref{item unit}. If $\mathscr F \in \Sh\lc^t(Y_\et)$ is connected, then the same goes for $j^*\mathscr F$, and $j_*j^*\mathscr F = \mathscr F$ by \autoref{Lem locally constant}\ref{item unit}. Thus, the action of $(\overline M\gp)^\vee$ on $j^*\mathscr F$ is trivial, as otherwise $i^*\mathscr F = \varnothing$ by the first computation above. The second computation shows that $\mathscr F$ comes from $\Sh\lc^t(Y_\et)$. This proves that $\pi_1^t(X) \to \pi_1^t(Y)$ is an isomorphism, and the same then goes for $\pi_1^t(Y) \to \pi_1^t(X)$ since the composition $\pi \circ i$ is the identity on $Y$.
\end{proof}

This gives the full version of \autoref{Lem BL} in the \'etale case:

\begin{Cor}\label{Cor toric}
Let $X$ be a split normal toric variety over a field $k$.
\begin{enumerate}
\item\label{item finite limit} A sheaf $\mathscr F \in \Sh(X_\et)$ is in $\Sh\cons[S]^t(X_\et)$ if and only if $\mathscr F$ is a finite limit of sheaves of the form $j_{s,*} \mathscr G$, where $j_s \colon X_s \hookrightarrow X$ is the locally closed immersion of some stratum and $\mathscr G \in \Sh\lc^t(X_s)$.
\item\label{item affine} If $X$ is affine and $i \colon Y \hookrightarrow X$ and $\pi \colon X \twoheadrightarrow Y$ are the maps defined in \ref{Par affine}, then the transformation $\pi_* \to i^*$ of \ref{Par section} gives an isomorphism of functors $\Sh\cons[S]^t(X_\et) \to \Sh\lc^t(Y_\et)$.
\end{enumerate}
\end{Cor}

\begin{proof}
\ref{item finite limit} If $t \in S$, we want to show that $j_t^*j_{s,*}\mathscr G \in \Sh\lc^t(X_{t,\et})$. If $t \leq s$, we may replace $X$ by $X_{\geq t} \cap X_{\leq s}$, and then \autoref{Lem toric} shows that $j_t^*j_{s,*} \mathscr G \in \Sh\lc^t(X_{t,\et})$. If $t \not\leq s$, we have $j_t^*j_{s,*} \mathscr G = *$, which is also in $\Sh\lc^t(X_{t,\et})$. As $\Sh\cons[S](X_\et)$ is stable under finite limits, this shows one inclusion. The converse follows by induction on $\lvert S \rvert$ in the same way as in \autoref{Lem BL}, using the glueing procedure explained in \cite[Ch.\ II, Thm.\ 3.10]{Milne}. Now \ref{item affine} follows from \autoref{Lem toric} since $i^*$ and $\pi_*$ preserve finite limits.
\end{proof}

Analogously to \autoref{Thm complex}, we get the following result for tame \'etale fundamental categories.

\begin{Thm}\label{Thm toric}
Let $k$ be a field, let $X$ be a split normal toric variety, and let $X \to S$ be the toric stratification. There is a canonical equivalence
\[
\Sh\cons[S]^t(X_\et) \simeq \Fun\cts\big(\Pi_1^t(X,S),\Fin\big),
\]
where $\Pi_1^t(X,S)$ is the profinitely enriched category whose objects are $s \in S$ and whose morphisms are
\[
\Hom(s,t) = \begin{cases} \pi_1^t(X_{\geq s}), & s \leq t,\\ \varnothing, & s \not\leq t. \end{cases} 
\]
Composition $\Hom(s,t) \times \Hom(r,s) \to \Hom(r,t)$ for $r \leq s \leq t$ is induced by multiplication in $\pi_1^t(X_{\geq r})$ under the map $\pi_1^t(X_{\geq s}) \to \pi_1^t(X_{\geq r})$. Moreover, $\pi_1^t(X_{\geq s}) \cong \pi_1^t(X_s)$ is the semidirect product $(\widehat{\mathbf Z}^{(p')}(1))^{\dim X_s} \rtimes \Gal(\bar k/k)$.
\end{Thm}

\begin{proof}\renewcommand{\qedsymbol}{}
Each stratum $X_s$ is isomorphic to $\Spec k[M_s]$ for some finite free group $M_s$, and has a tame universal cover $\widetilde{X_s} \to X_s$ given by $\Spec k[M_s \otimes \mathbf Z_{p'}]$. Let $\bar s = 1 \in X_s(\bar k)$ be the identity geometric point, and $F_{\bar s} \colon \Sh\cons[S]^t(X_\et) \to \Fin$ the corresponding fibre functor. The general exodromy theorem \cite[2.21]{vDdB} gives the first statement, where $\Pi_1^t(X,S)$ is the full subcategory of $\Fun\big(\Sh\cons[S]^t(X_\et),\Fin\big)$ on the functors $F_{\bar s}$. It remains to compute the $\Hom$ sets in $\Pi_1^t(X,S)$. 

If $s \in S$, then the fibre functor $F_{\bar s} \colon \Sh\lc^t(X_{s,\et}) \to \Fin$ is canonically pro-represented by the universal cover $\widetilde{X_s} \in \Pro\fin(\Sh\lc(X_\et))$. Write $i \colon X_s \to X_{\geq s}$ for the closed immersion, and $j = j_{\geq s} \colon X_{\geq s} \to X$ for the open immersion. By \autoref{Cor toric}, the pullback functor $i^* \colon \Sh\cons[S]^t(X_{\geq s,\et}) \to \Sh\lc^t(X_{s,\et})$ has $\pi^*$ as left adjoint. Moreover, $j_!$ is left adjoint to $j^*$. Thus, for every $\mathscr F \in \Sh\cons[S]^t(X_\et)$, we have
\[
F_{\bar s}(\mathscr F) = F_{\bar s}(i^*j^*\mathscr F) = \Hom_{X_s}\big(\widetilde{X_s},i^*j^*\mathscr F\big) = \Hom_X\big(j_!\pi^*\widetilde{X_s},\mathscr F\big),
\]
so $F_{\bar s}$ is pro-represented by $j_!\pi^*\widetilde{X_s}$. We saw in \autoref{Lem toric} that $\pi \colon X_{\geq s} \to X_s$ induces an isomorphism on tame fundamental groups, so $\pi^*\widetilde{X_s}$ is the universal cover $\widetilde{X_{\geq s}}$. Now we conclude by the Yoneda lemma: if $s, t \in P$, then
\begin{align*}
\Hom\big(F_{\bar s},F_{\bar t}\big) = \Hom\big(h_{j_{\geq s,!}\widetilde{X_{\geq s}}},F_{\bar t}\big) = F_{\bar t}\big(j_{\geq s,!}\widetilde{X_{\geq s}}\big) = \begin{cases}\pi_1^t(X_{\geq s}), & s \leq t, \\ \varnothing, & s \not\leq t. \end{cases}\tag*{$\square$}
\end{align*}
\end{proof}\vspace{-1.3em}

\begin{Par}
In the topological setting, we actually have $\Pi_\infty(X,S) = \Pi_1(X,S)$ by \autoref{Rmk infty}. We expect the statement $\Pi_\infty^t(X,S) = \Pi_1^t(X,S)$ to hold in the \'etale case as well. However, there is currently not even a definition of stratified tame \'etale homotopy types, nor an exodromy theorem, so it is not so clear how to state or prove this.

For the $1$-categorical version, we used that tame locally constant \'etale sheaves coincide with the full subcategory of locally constant \'etale sheaves that are trivialised on a tame cover. This cannot be true $\infty$-categorically: for a finite group $G$, the mapping space $\pi_0\Map_{X_\et}(*,BG)$ between two \emph{constant} sheaves computes all $G$-torsors, not just the tamely ramified ones. Thus, the inclusion from tame sheaves of spaces to \'etale sheaves of spaces is not fully faithful. This is not surprising, as tame homotopy types should be constructed in a much more subtle way \cite[\S5]{HubnerSchmidt}.
\end{Par}

\phantomsection
\bibliographystyle{alphaurledit}
{\footnotesize\bibliography{Toric.bib}}

\end{document}